\documentclass[twoside]{amsart}
\usepackage{latexsym}
\usepackage{amssymb,amsmath,amsopn}
\usepackage[dvips]{graphicx}   
\usepackage{color,epsfig}      

               {\begin{list}{}{\leftmargin#1\rightmargin#2}\item{}}%
               {\end{list}}
 
\newtheorem{thm}{Theorem}

\newtheorem{prop}{Proposition}
\theoremstyle{definition}
\newtheorem{defn}{Definition}
\newtheorem{rem}{Remark}

\renewcommand{\Re}{\mathbb R}

\newcommand{\BB}{\mathbf B}

\renewcommand{\S}{\mathbb{S}}

\def\bea{\begin{eqnarray}}
\def\eea{\end{eqnarray}}

\DeclareMathOperator{\bd}{bd}

\DeclareMathOperator{\diam}{diam}

\begin{document}
\title[The Eikonal decreases the isoperimetric quotient]{The isoperimetric  quotient of  a convex body decreases monotonically under the Eikonal abrasion model}
\author[G. Domokos and Z. L\'angi] {G\'abor Domokos and Zsolt L\'angi}
\address{G\'abor Domokos, MTA-BME Morphodynamics Research Group and Dept. of Mechanics, Materials and Structures, Budapest University of Technology,
M\H uegyetem rakpart 1-3., Budapest, Hungary, 1111}
\email{domokos@iit.bme.hu}
\address{Zsolt L\'angi, MTA-BME Morphodynamics Research Group and Dept. of Geometry, Budapest University of Technology,
Egry J\'ozsef utca 1., Budapest, Hungary, 1111}
\email{zlangi @math.bme.hu}

\subjclass[2010]{35Q85, 35Q86, 52A39}
\keywords{Eikonal equation, abrasion, elongation, asteroid}

\begin{abstract}
We show that under the Eikonal abrasion model, prescribing uniform normal speed in the direction of the inward surface normal, the isoperimetric quotient of a convex shape is decreasing monotonically. 
\end{abstract}
\maketitle

\section{Introduction}
The Eikonal equation is a nonlinear partial differential equation describing, among other things, the evolution of surfaces. In this paper we study a special case of this equation restricted to convex surfaces which, in compact notation, may be written for the evolution $K(t)$ of a convex body  $K \in \Re^d, d>1$ as
\begin{equation}
\label{e1}
v(p)=1,
\end{equation}
where $p \in \bd K$ is a point on the boundary of $K$ and  $v(p)$ is the speed by which $p$ moves in the direction of the inward surface normal.  This equation may also be written in the traditional PDE notation, for example,  in the plane, for the radial evolution of a curve $r(\varphi),$  Equation (\ref{e1}) is equivalent to
\begin{equation}\label{PDE}
\frac{\partial r}{\partial t}= - \frac{1}{r}\sqrt{r^2+\left(\frac{\partial r}{\partial \varphi}\right)^2}.
\end{equation}
The Eikonal equation has broad applications, ranging from geometric optics to abrasion. Two possible geometric interpretations of (\ref{e1}) appear in the literature  (see \cite{Arnold}, Theorem 1 and \cite{bloore}), depending on which application is targeted by the model. We describe these interpretations below and illustrate them in Figure~\ref{fig0}.
                                          
In the \emph{Eikonal wavefront model},  starting with the boundary of a convex body $K(0)$, one obtains the evolving hypersurface at time $t$ by translating every \emph{point} of its boundary in the direction of the inward surface normal by a vector of length $t$
(cf. Figure \ref{fig0} (A)). If $K(0)$ is a smooth, convex body with minimal curvature radius $r_{\min}(0)$ (i.e. the reciprocal of the maximal principal curvature) then  in the Eikonal wavefront model the evolving hypersurface  will exhibit its first singularity at $t=r_{\min}(0)$ and for $t>r_{\min}(0)$ it will develop self-intersecting, non-convex parts.

In the \emph{Eikonal abrasion model}, starting with the boundary of a convex body $K(0)$, one obtains the evolving hypersurface at time $t$ by translating  \emph{the supporting half space} at every point of its boundary in the direction of the inward surface normal by a vector of length $t$
(cf. Figure \ref{fig0} (B)). Obviously, in this model the evolving hypersurface $K(t)$ will remain convex, however, initially smooth shapes will also develop singularities. The first  such singularity appears, similarly to the wavefront model, at  $t=r_{\min}(0)$. The singularities of the evolving hypersurface in the Eikonal abrasion model correspond to the self-intersections  in the Eikonal wavefront model. 

\begin{figure}
\begin{center}
\includegraphics[width=12cm]{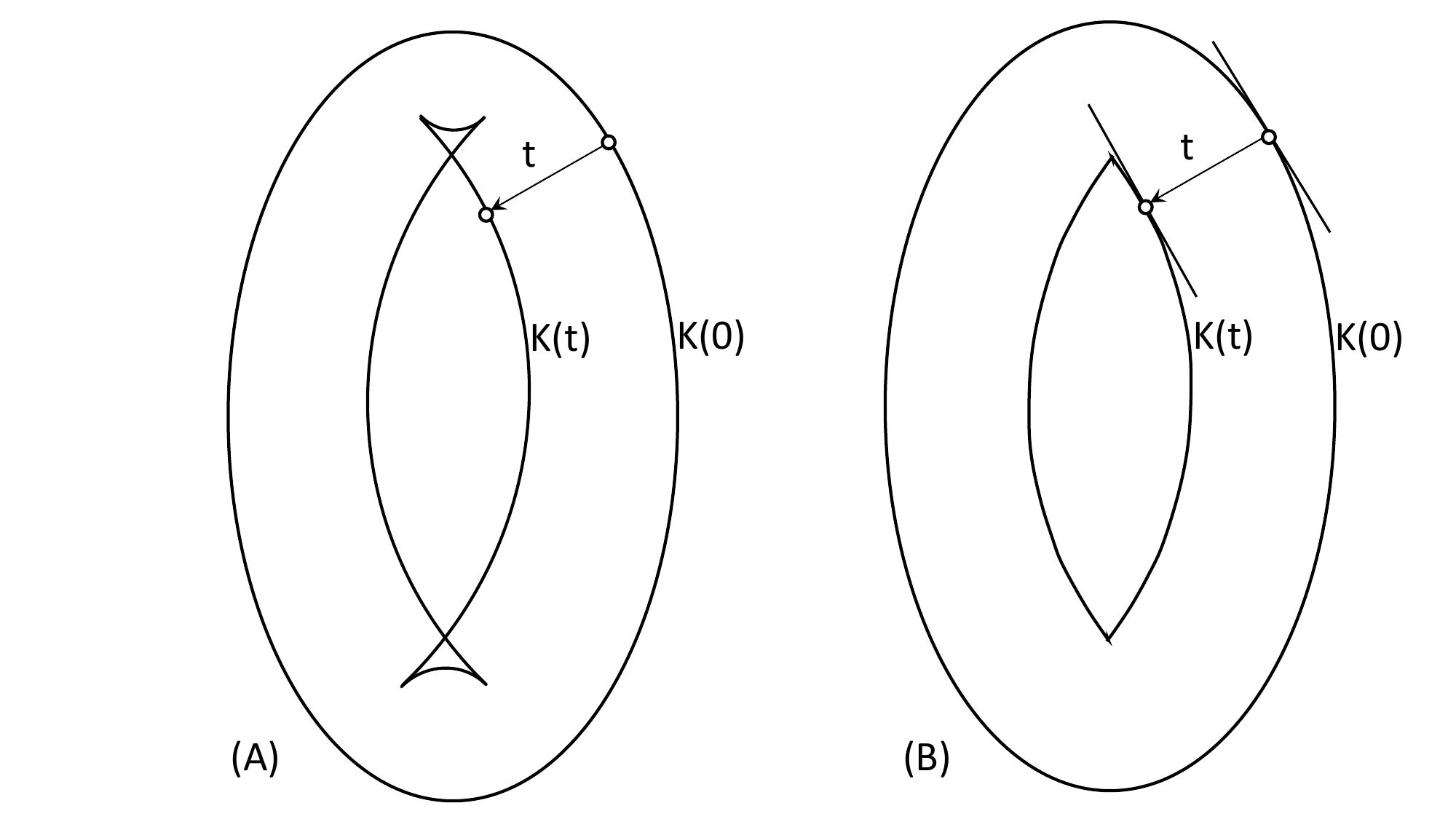}
\caption{Alternative interpretations of the Eikonal equation (\ref{e1}): (A) The Eikonal Wavefront Model and (B) the Eikonal Abrasion Model.
\label{fig0}}
\end{center}
\end{figure}
                                                                                  
The current paper focuses on the Eikonal abrasion model  (for details about global geometrical evolution in the wavefront model see, for example, \cite{Knill}). 
In an abrasion process an abrading object $K$ with surface area $A(K)$ is bombarded by abraders $k$ with surface area $A(k)$ arriving from directions uniformly distributed over the sphere and the surface ratio $\mu = A(k)/A(K)$ is a characteristic parameter of the process \cite{bloore}.
In this context, the Eikonal equation has been introduced by Bloore \cite{bloore} as a model for the limit $\mu \to 0$. This model can explain, among other things, the shapes of desert rocks abraded by wind-blown sand \cite{Laity} and also the shapes of asteroids \cite{Domokosetal}, which evolve under many small impacts from micrometeorites. In the latter application a geometric feature of the Eikonal model
played a key role: it has been observed \cite{oumuamua} that evolution under the Eikonal abrasion model  decreases the roundedness (i.e. increases the thinness/flatness) of a shape.   This observation offered so far the only natural explanation for the grossly elongated shape of the first observed interstellar asteroid `Oumuamua \cite{oumuamua}. The aim of the present note is to give a mathematical verification of this observation. We prove that in this model the \emph{isoperimetric quotient}, which is one of the most broadly used  measures of resemblance to a sphere, decreases as a function of time.
There has been great interest in proving the monotonicity of the isoperimetric quotient and related quantities under geometric partial differential equations, most notably, curvature-driven evolution equations \cite{gage, Hamilton, Huisken}. Even though the Eikonal equation (\ref{e1}) is fairly different from these flows, 
the evolution of the isoperimetric quotient under (\ref{e1}) still looks like an attractive problem: as we will point out in Subsection \ref{other}, in mathematical models of abrasion the Eikonal model is coupled with curvature-driven terms.


In order to formulate our main result, we introduce some notation for the concepts described above.  First, by $\BB^d$ we denote the closed unit ball with the origin $o$ as its center, and set $\S^{d-1}=\bd \BB^d$.
Let $K$ be a convex body in $\Re^d$; that is, a compact, convex set with nonempty interior. For any $t \geq 0$, we denote by $h_K : \S^{d-1} \to \Re$ the \emph{support function} of $K$, and let $H_K(u,t)$ be the closed half space defined by $\{ x \in \Re^d : \langle x,u \rangle \leq h_K(u)-t \}$. Furthermore, for any $t>0$, we set $K(t) = \bigcap\limits_{u \in S^{d-1}} H_K(u,t)$. Observe that if $r(K)$ is the radius of a largest ball contained in $K$, then for any $0\leq t < r(K)$ $K(t)$ is a convex body, for $t=r(K)$, $K(t)$ is a compact, convex set with no interior point, and for $t > r(K)$, $K(t)=\emptyset$. Following \cite{pisanski}, we define the isoperimetric quotient of any convex body $M$ as $I(M)=\frac{V(M)}{(A(M))^{\frac{d}{d-1}}}$ where $V(\cdot)$ and $A(\cdot)$ denotes $d$-dimensional volume and $(d-1)$-dimensional surface area, respectively.
Our main result is as follows.

\begin{thm}\label{thm:isoperimetric_quotient}
For any convex body $K$, either $I(K(t))$ is strictly decreasing on $[0,r(K))$, or
there is some value $t^{\star} \in [0,r(K))$ such that $I(K(t))$ is strictly decreasing on $[0,t^{\star}]$ and is a constant on $[t^{\star},r(K))$. Furthermore, in the latter case, for any $t > t^{\star}$, $K(t)$ is homothetic to $K(t^{\star})$.
\end{thm}

While our result highlights a fundamental property of the nonlinear PDE (\ref{e1}), our proof will rely entirely on ideas and concepts from convex geometry.
In particular, as one can see in Section~\ref{sec:proof}, the examined problem, and also the applied tools, are closely related to those used in \cite{larson}, where,
by an elegant method, the author gave a sharp lower estimate on the surface area of an inner parallel body of a convex body. These bodies also play a central role in the current paper:  Equation (\ref{e1})  can be regarded as a \emph{map} from a convex body to  one of its inner parallel bodies where time plays the role of the distance of this body from the original one. In this paper we study geometric properties of this map.
The proof relies on an extensive use of mixed volumes and inequalities related to them. For the definitions and descriptions of these tools, the reader is referred to the books \cite{Gruber, schneider} or the paper \cite{larson}.
Some elements of our proof can also be found in \cite[Section 7.2]{schneider} and in \cite{stacho}.

\section{Proof of Theorem~\ref{thm:isoperimetric_quotient}}\label{sec:proof}

 During the proof, for any convex body $M$ and boundary point $p \in \bd M$ we denote the set of outer unit normal vectors of $M$ at $p$ by $N_p(M) \subseteq \S^{d-1}$.
Note that this set is a spherically convex,  compact set, and hence, in particular, it is contained in an open hemisphere of $\S^{d-1}$.
The set of smooth (or regular) points of $M$ (i.e. the set of boundary points $p$ of $M$ such that $N_p(M)$ is a singleton) is denoted by $R(M)$.
Finally, we set  $N(M)= \bigcup_{p \in R(M)} N_p(M)$, and
\begin{equation}\label{eq:P}
F(M) = \{ x \in \Re^d: \langle x,u \rangle \leq 1 \hbox{ for every } u \in N(M)\}.
\end{equation}
This set is called the \emph{form body} of $M$ (cf. \cite{larson} and  \cite{schneider}), where the fact that it is indeed a convex body
is an easy consequence of Alexandrov's theorem (cf. Theorem 5.4 in \cite{Gruber}), stating that $\bd M$ is smooth (twice differentiable) at almost every point.


Now we prove Theorem~\ref{thm:isoperimetric_quotient}.
As volume and surface area are continuous with respect to Hausdorff distance and are strictly increasing with respect to containment in the family of convex bodies, both $V(K(t))$ and $A(K(t))$ are positive continuous, strictly decreasing functions of $t$ on the interval $[0,r(K))$. This implies that $I(K(t))$ is also 
continuous on this interval.

Consider some $0 < t_0 < r(K)$. For brevity, we set $K(t_0)=K_0$,  $N(K(t_0))=N_0$, $F(K(t_0))=F_0$ and $I(K(t))=I(t)$.
Let $p \in \bd K_0$. By the definition of $K(t)$, for any $0 \leq t \leq t_0$ the minimum of the distances of $p$ from the supporting hyperplanes of $\bd K(t)$ is equal to $t_0-t$. Since this minimum is attained at some point $q \in \bd K(t)$, the ball $p + (t-t_0) \BB^d$ touches $\bd K(t)$ from inside.
Thus, $K_0 + (t_0-t) \BB^d \subseteq K(t)$. On the other hand, we also have $h_{K(t)}(u) = h_{K_0}(u)+(t_0-t)$ for any $u \in N_0$, which yields that
$K(t) \subseteq K_0 + (t_0-t)F_0$.
 We remark that the geometric arguments leading to these two containment relations can be found in a comprehensive form in \cite{larson}.
We note that if $K_0$ is smooth, then so is $K(t)$ for every $0 \leq t \leq t_0$, implying that $K(t)=K_0 + (t_0-t) \BB^d$ in this case. More generally, $N(K(t))$ decreases and $F(K(t))$ increases in time with respect to containment; that is, for any $0 \leq t_1 < t_2 < r(K)$ we have $N(K(t_2)) \subseteq N(K(t_1))$
and $F(K(t_1)) \subseteq F(K(t_2))$.

Set $L(t) = K_0 + (t_0-t) \BB^d$, and $M(t) = K_0 + (t_0-t) F_0$. By Minkowski's theorem on mixed volumes \cite{schneider}, we have
\begin{equation}\label{eq:mixed_volumes}
V(L(t))= \sum_{j=0}^d \binom{d}{j} (t-t_0)^j W_j(K_0), \hbox{ and }
V(M(t))= \sum_{j=0}^d \binom{d}{j} (t-t_0)^j V_j(K_0),
\end{equation}
where $W_j(K_0)$ is the $j$th quermassintegral of $K_0$, and we denote the mixed volume $V(\overbrace{K_0,\ldots,K_0}^{d-j}, \overbrace{F_0,\ldots,F_0}^{j}  )$ by $V_j(K_0)$.

Observe that $W_0(K_0)=V_0(K_0)=V(K_0))$, and that $d W_1(K_0) = A(K_0)$.
We show that $d V_1(K_0) = A(K_0)$ as well.
Since both mixed volumes and surface area are continuous with respect to Hausdorff distance, it suffices to prove this equality for polytopes, and thus, assume for the moment that $K_0$ is a polytope. In this case $F_0$ is the polytope, circumscribed about $\BB^d$, whose outer unit facet normal vectors coincide with those of $K_0$. Thus, $M(t)$ can be decomposed into $K_0$, cylinders of height $t_0-t$ with the facets of $K_0$ as bases, and sets in the $\rho (t_0-t)$-neighborhood of the $(d-2)$-faces of $K(t_0)$, where $\rho =  \diam F_0$.
The volume of this set is
\[
V(M(t))=V(K_0) + (t_0-t) A(K_0) + O((t_0-t)^2),
\]
implying $-dV_1(K_0) = \left. \frac{d}{dt} V(M(t)) \right|_{t=t_0} = -A(K(t_0))$.

Let us define the quantity 
\[
\underline{I}(t) = \frac{V(L(t))}{A(M(t))^{\frac{d}{d-1}}}.
\]
We note that $\underline{I}(t)$ depends on $t_0$ and it is defined only for $0 \leq t \leq t_0$. Furthermore,
as both volume and surface area are strictly increasing with respect to inclusion, we have $\underline{I}(t) \leq I(t)$.
Differentiating this quantity, the formulas in (\ref{eq:mixed_volumes}) and their connection with $A(K_0)$ yields that
\[
\underline{I}'_{-}(t_0) = -\frac{d^2}{A(K_0)^{\frac{2d-1}{d-1}}} 
\left( V_1(K_0)^2-V_0(K_0)V_2(K_0)\right),
\]
which is not positive by the second inequality of Minkowski \cite{Gruber}.

We show that if $\underline{I}'_{-}(t_0)=0$, then $K_0$ is homothetic to $F_0$.
By Theorem 7.6.19 in \cite{schneider}, since both $K_0$ and $F_0$ are $d$-dimensional, $(V_1(K_0))^2 = V_0(K_0)V_2(K_0)$ implies that $K_0$ is homothetic to a $(d-2)$-tangential body of $P$. More specifically, $K_0$ has a homothetic copy $K'$ such that $F_0 \subseteq K'$, and every supporting hyperplane of $K'$ that does not support $F_0$ contains only $(d-3)$-singular  (cf. \cite[Section 2.2]{schneider}), or in particular, singular points of $K'$. Hence, every supporting hyperplane of $K'$ that contains a smooth point of $\bd K'$ supports $F_0$ as well. Thus, the definition of $F_0$ and the relation $F_0 \subseteq K'$ yields $F_0=K'$.
This means that if $\underline{I}'_{-}(t_0) = 0$, then $F_0$ is homothetic to $K_0$. Note that the reversed statement also holds: if $F_0$ is homothetic to $K_0$, then $\underline{I}'_{-}(t_0) = 0$, and even more, in this case $K(t)$ is homothetic to $K_0$ for any $t > t_0$.

Let $t^{\star}$ denote the smallest value of $t$ such that $\underline{I}'_{-}(t) = 0$. Then $K(t)$ is homothetic to $K(t^{\star})$ for any $t \in [t^{\star},r(K))$, and $I(t)$ is a constant on this interval. To finish the proof we need to show that $I(t)$ strictly decreases on $[0,t^{\star}]$.

Since $\underline{I}'_{-}(t_0) < 0$ for any $0 < t_0 < t^{\star}$, for any such $t_0$ there is some $\varepsilon = \varepsilon(K,t_0) > 0$ such that $I(t) \geq \underline{I}(t) > I(t_0)$ for all $t \in (t_0-\varepsilon,t_0)$; that is, the function $I(t)$ is locally strictly decreasing from the left at every point.
This and the continuity of $I(t)$ implies that $I(t)$ strictly decreases on this interval.
Indeed, suppose for contradiction that for some $t_1 < t_2$ we have $I(t_1) \leq I(t_2)$.
By continuity, $I(t)$ attains its global maximum on $[t_1,t_2]$ at some $t' \in [t_1,t_2]$. Clearly, since $I(t)$ is locally strictly decreasing from the left at $t'$ it follows that $t'=t_1$ and $I(t_1) > I(t_2)$, a contradiction.
\qed

\section{Discussion}
In this section we will point out connections to related results in convex geometry and also 
put our result in a broader setting to show that it may have implications beyond the specific PDE (\ref{e1}).

\subsection{Connection to Lindel\"of's Theorem and another type of asphericity}
While our paper is primarily aimed to prove properties of the Eikonal abrasion model (\ref{e1}), the proof of Theorem \ref{thm:isoperimetric_quotient} also implies a modest improvement of Lindel\"of's Theorem \cite{lindelof}. 
Let $K$ be a convex $d$-polytope with outer unit normal vectors $u_1,u_2, \ldots,u_n$ of its facets. Then $F(K)$ is the polytope with the same vectors as outer facet unit normal vectors, circumscribed about the unit ball $\BB^d$. For this case, Theorem~\ref{thm:isoperimetric_quotient} states that $I(K + tF(K))$ with $t \in [0,\infty )$ (or, since homothety does not change the isoperimetric ratio, $( (1-t) K + t F (K) )$ with $t \in [0,1]$), is a { strictly increasing function of $t$, unless $K$ is homothetic to $F(K)$.
This special case is a stronger form of Lindel\"of's theorem \cite{lindelof} stating that among convex polytopes with given outer facet unit normal vectors, those circumscribed about a ball has maximal isoperimetric ratio.

 This connection leads us to the following definition.

\begin{defn}\label{defn:envelope}
Let $K$ be a convex body.  Let the convex body $E(K)$ be the intersection of the closed supporting half spaces whose boundary touches each largest inscribed ball of $K$. We call this closed, convex set the \emph{envelope of $K$ under Equation (\ref{e1})}, or shortly, the \emph{envelope of $K$}.
\end{defn}

We remark that $E(K)$ is a convex body if, and only if there is a unique largest ball inscribed in $K$, $F(K)$ is contained in a translate of $\frac{1}{r(K)} E(K)$, and that the quantity $t^{\star}$ appearing in Theorem~\ref{thm:isoperimetric_quotient} has the property that this is the smallest value of $t$ such that this containment is not strict. In the spirit of \cite{dudov}, we will use the following

\begin{defn}\label{def1}
Let $K$ be a convex body and let $r(K)$ and $R(K)$ denote the radius of a largest inscribed and the smallest circumscribed sphere of $K$, respectively. Then the \emph{asphericity} of $K$ is given by $\alpha(K)=\frac{R(K)}{r(K)}$.
\end{defn}

\begin{prop}\label{lem1}
Under the Eikonal abrasion model given by Equation (\ref{e1}), acting on the convex body $K(t)$, $\alpha(K(t))$ is an increasing function of $t$ on $[0,r(K))$.
\end{prop}                                                                                                                    

\begin{proof}
Recall the definition of $K(t)$ in the  Introduction:
\[
K(t) = \bigcap\limits_{u \in S^{d-1}} H_K(u,t),
\]
where $H_K(u,t)$ is the closed half space defined by $\{ x \in \Re^d : \langle x,u \rangle \leq h_K(u)-t \}$.
Since changing the origin only translates $K(t)$, we may assume that $o$ is the center of a largest inscribed ball in $K$.
Then $r(K) = \min_{u \in S^{d-1}} h_K(u)$, and it is easy to see  (for a formal proof, see Lemma 1.4 in \cite{larson}) that $r(K(t))= r(K)-t$ for all values $t \in [0,r(K))$.

Note that for any $t \in [0,r(K))$, the radius of a largest inscribed ball of $\frac{r(K)}{r(K(t))} K(t)$ is $r(K)$.
On the other hand, for any $u \in \S^{d-1}$, the inequality $0 < r(K) \leq h_K(u)$ yields that $\frac{h_K(u)}{r(K)} \leq \frac{h_K(u)-t}{r(K)-t}$ for all $0< t < r(K)$,
implying that $K \subseteq \frac{r(K)}{r(K(t))} K(t)$. Thus, $\alpha(K) \leq \alpha(K(t))$ for every value of $t$.  Since the same argument implies also that $K(t_1) \subseteq K(t_2)$ for all $0 \leq t_1 \leq t_2 < r(K)$, $\alpha(K(t))$ is an increasing function of $t$.
\end{proof}

Observe that the proof of Proposition \ref{lem1} indicates that if $E(K)$ is bounded, then as $t$ tends to $r(K)$, $K(t)$ `approaches' $E(K)$ up to homothety.
The proof also illustrates that the monotonicity of the asphericity given in Definition \ref{def1} is much easier to establish than the monotonicity of the isoperimetric quotient (which we did in Theorem \ref{thm:isoperimetric_quotient}). Nevertheless, as we will point out below, the latter result has broader consequences in abrasion models.

\begin{rem}
Equation (\ref{e1}) can be generalized for an arbitrary strictly convex, smooth finite dimensional normed space with unit ball $C$ in the following way:
Following \cite[Definition 3.2.2.]{thompson}, we say that $v \in \bd C$ is \emph{normal} to $\bd K$ at $p$ if a supporting hyperplane of $K$ at $p$ is parallel to the supporting hyperplane of $C$ at $v$. In particular, we say that if $p-v+C$ touches $K$ then $v$ is an \emph{inward unit normal} of $K$ at $p$, and if $p+v+K$ touches $K$ then it is an \emph{outward unit normal} of $K$.
Consider the evolution defined by
\begin{equation}
\label{e:normed}
v(p)=1,
\end{equation}
where $ p \in \bd K$ and $v(p) \in \bd C$ is the speed by which $p$ moves in the direction of the inward unit normal of $K$ at $p$. 
Using this equation and the \emph{support function} of $K$ with respect to the norm of $C$ (cf. e.g. \cite{Spirova}), we may define the \emph{normed} Eikonal abrasion model with respect to the norm of $C$ like in the Euclidean case.
A straightforward modification of our proof yields that under this Eikonal abrasion model, the quantity
\[
\frac{V(K(t))}{\left( V(\overbrace{K(t),\ldots,K(t)}^{d-1},C) \right)^{\frac{d}{d-1}}}
\]
decreases monotonically for any initial convex body $K(0)$.
Here $V(\cdot)$ denotes $d$-dimensional Euclidean volume, and we remark that every finite dimensional normed space can be equipped by a Haar measure and that this measure is unique up to multiplication of the standard Lebesgue measure $V(\cdot)$ by a scalar \cite{langi, paiva}.
\end{rem}

\subsection{Connection to other  surface evolution models}\label{other}
As  remarked in the introduction, the Eikonal model corresponds to  the rather special limit $\mu \to 0$  in collisional abrasion when the relative size of the incoming abraders and the abraded object approaches zero.
The other extreme limit is $\mu \to \infty$, when
the incoming particles are very large, in this case they can be modeled by hyperplanes. This case has been first treated in the classic paper by Firey \cite{firey}, who arrived at the formula
\begin{equation}\label{firey}
v=c\kappa,
\end{equation}
where $\kappa$ is the curvature  at the boundary point of the abrading particle if (\ref{firey}) is interpreted in two dimensions, it is the Gaussian curvature in the  $3$-dimensional case and $c$ is a scalar coefficient. 
The  planar version of (\ref{firey}) is called the \emph{curve shortening flow}, under which, by a result of  Gage \cite{gage}, isoperimetric quotient is monotonically increasing.
Bloore \cite{bloore} generalized Firey's model by admitting incoming particles with arbitrary size. Quite surprisingly, the PDE describing the general case turned out to be a linear combination of the two extreme scenarios. In two dimensions, Bloore's equation is
\begin{equation}\label{bloore}
v=1+c\kappa,
\end{equation}
where $\kappa$ is the curvature and $c$ is the (normalized) perimeter of the incoming particles. 
Physical intuition suggests that in Equation (\ref{bloore}) the two additive terms represent geometrically opposite, competing effects, however, it is not easy to formalize this observation. Our Theorem
\ref{thm:isoperimetric_quotient}, combined with
Gage's result \cite{gage} shows that as far as the evolution of the isoperimetric ratio is concerned, the two terms are indeed acting in  opposite directions.

The Eikonal equation (\ref{e1}) can be also interpreted as a speed defined for the \emph{outward} surface normal and this leads to the time-reversed Eikonal model where the wavefront and surface evolution interpretations do not need to be distinguished, however, the initial shape $K(0)$ has to be smooth, otherwise the flow is not unique. Unlike the inward flow, the outward Eikonal flow preserves the smoothness of shapes. It is easy to see that in the $t\to \infty$ limit the outward flow converges to the sphere and as a trivial consequence of Theorem \ref{thm:isoperimetric_quotient} we can also see that $I(t)$ is growing monotonically under this flow.  The time-reversed Eikonal flow appears in
one of the most broadly used surface growth equations in soft condensed matter physics, the Kardar-Parisi-Zhang (KPZ) equation \cite{kardar}. 
The deterministic version of the KPZ equation can be written \cite[Equation 29]{marsili} in two dimensions as
\begin{equation}\label{KPZ3}
v=-1+c\kappa,
\end{equation}
which is reminiscent of the Bloore flow (\ref{bloore}), however, the additive terms have opposite sign, so physical intuition suggests that
for any smooth, convex curve $K(0)$ evolving under (\ref{KPZ3}) the isoperimetric quotient $I(t)$ is a monotonically
increasing function. We also remark that the Eikonal abrasion model appears to be the only one among collisional abrasion models where the isoperimetric ratio decreases. However, among models for frictional abrasion \cite{matgeo} one can find analogous evolutions.

\subsection{Evolution of the number of critical points}
One can obtain interesting information about geometric evolutions by identifying quantities which evolve monotonically under the  flow. Such quantities may be either real-valued (prominent  examples are versions of the isoperimetric quotient \cite{Hamilton, gage}),  or alternatively, they may be integer-valued.  Integer-valued quantities have been investigated mostly in connection of the heat equation and other flows related to image processing \cite{mumford, Koenderink} and  a prominent example  is the number $N(t)$ of critical points. Here the evolving hypersurface is defined by the scalar distance $r$ measured either
from a hyperplane (in orthogonal coordinates) or from a  fixed point (in polar coordinates,
e.g. such as in Equation (\ref{PDE})). These hypersurfaces correspond to graphs of real valued functions $r$ defined on $\Re^{d-1}$ or on $\S^{d-1}$, respectively. Then the points satisfying $\triangledown r=0$ are called \emph{critical}.  The function $N(t)$ is constant for almost all values of $t$ and we call it
monotonic if it has jumps only in one direction. In the heat equation it was believed for an extended
period of  time that $N(t)$ is decreasing monotonically, nevertheless, Damon \cite{Damon} showed counterexamples. On the other hand, for the curve-shortening flow (\ref{firey}) Grayson proved \cite{Grayson} that $N(t)$ is indeed decreasing monotonically. 
As illustrated in Subsection \ref{other},
the Eikonal abrasion model (\ref{e1}) is closely related to the curve shortening flow (\ref{firey}) via Bloore's general abrasion model (\ref{bloore}), so establishing a monotonicity result for $N(t)$ under (\ref{e1}) may help to clarify some aspects of shape evolution under (\ref{bloore}).  Unlike the heat equation and curvature-driven flows, in the Eikonal model the evolution of $N(t)$ is very easy to establish. Below we refer to the Eikonal equation written in polar coordinates for the distance function $r$,  which, in two dimensions, is given in (\ref{PDE}), or the geometric observations made in the paragraph before (\ref{eq:mixed_volumes}).
We note that Proposition~\ref{crit} has been proved under some smoothness condition in Lemma 2 of \cite{Dommon}.

\begin{prop}\label{crit}
In the Eikonal wavefront model and in the time-reversed Eikonal model $N(t)$ is constant. In the Eikonal abrasion model $N(t)$ is decreasing monotonically.
\end{prop}

\begin{proof}
The statement regarding the Eikonal wavefront model can be seen immediately as we realize that under (\ref{e1}), each point is moving along the associated inward surface normal. Points are critical if that normal passes through the reference point, so each point which is critical will remain critical and each point which
is non-critical will remain non-critical under the flow (\ref{e1}). The same observation applies to the time-reversed Eikonal model. In the Eikonal wavefront model and in the time-reversed Eikonal model points travel on infinite (straight)  trajectories, in the Eikonal abrasion model the straight trajectory is terminated after finite time. We know that the longest trajectories are of length $t=r(K)$.  In a non-degenerate case, that is, assuming that $K$ has finitely many critical points, there are finitely many trajectories we need to keep track of. Whenever a trajectory of a critical point terminates, $N(t)$ will decrease.
\end{proof}

Note that in the planar case, the first part of the statement also follows by the local analysis of codimension one saddle-node bifurcations of critical points of $r(\varphi)$ evolving under (\ref{PDE}). For details of this type of analysis see \cite{Dommon}.


\begin{rem}
Proposition \ref{crit} applies if the distance is measured from any fixed reference point. In physical applications often the center of mass is used as reference.
As this center, in general, may vary in time, the  monotonicity of critical points does not follow from our argument; in \cite{Dommon} an explicit counterexample is shown for this case.
\end{rem}

\subsection{Nonlinear theory}
Our analysis so far concerned the monotonicity of $I(t)$. Now we show that, at least in a special case, one may obtain information even on concavity.

\begin{prop}\label{prop:concavity}
If $K$ is a smooth plane convex body, different from a circle and with minimal curvature radius $r_{\min}(0)$, then on the interval $[0,r_{\min}(0)]$, $I(t)$ is a strictly concave function.
\end{prop}

\begin{proof}
In this case we have $L(t) = K(t)= M(t)$, $\underline{I}(t)=I(t)$ and thus,
\[
I''(t_0)= 
\frac{24 \pi^2}{(A(K_0)^2}
\left( \frac{V(K_0)}{A(K_0)^2} - \frac{1}{4\pi} \right),
\]
which, by the isoperimetric inequality, is not positive, and if $K_0$ is not a Euclidean disk, then it is negative.
\end{proof}

Whether or not Proposition~\ref{prop:concavity} remains valid in a more general setting, is an open question.
In \cite{oumuamua} the evolution of \emph{axis ratios} has been studied numerically. The \emph{axes} of a convex body have been defined as its largest diameter $a$, its largest width $b$
orthogonal to $a$ and its largest width $c$ orthogonal both to $a$ and $b$. The corresponding axis ratios are $c/a \leq b/a$. Needless to say, these definitions do not always lead to unique axes, however, this
is the definition broadly used in the geological literature \cite{blott}.
In \cite{oumuamua} it has been observed in the context of numerical simulations that the evolution of axis ratios is also accelerating in case of  $3$-dimensional objects
and this suggests that a result analogous to Proposition \ref{prop:concavity} could be established in three dimensions.

\section{Acknowledgment}
The authors thank L\'aszl\'o Lov\'asz for motivating this research by asking the proper question, and an anonymous referee for many helpful remarks that improved the quality of the paper. This research has been supported by the Hungarian Research Fund (OTKA) grant 119245.

\end{document}